\documentclass{amsart}
\usepackage{color}

\newtheorem{theorem}{Theorem}[section]
\newtheorem{lemma}[theorem]{Lemma}

\theoremstyle{definition}
\newtheorem{definition}[theorem]{Definition}

\newtheorem{proposition}[theorem]{Proposition}
\newtheorem{corollary}[theorem]{Corollary}
\theoremstyle{remark}

\input xypic
\input xy
\xyoption{all}

\numberwithin{equation}{section}

\begin{document}

\title{Homotopies of Lie Crossed Module Morphisms }


\author{\.{I}.\.{I}lker Ak\c{c}a}
\address{}
\curraddr{Department of Mathematics and Computer Science, Art and Science Faculty, Eski\c{s}ehir Osmangazi University, 26480, Eski\c{s}ehir, Turkey}
\email{iakca@ogu.edu.tr}
\thanks{}

\author{Yavuz Sidal }
\address{ }
\curraddr{Isiklar Air Force High School , Maths Teacher, 16039, Bursa/TURKEY}
\email{yavuz.sidal@yahoo.com}

\thanks{}



\date{}

\dedicatory{}

\begin{abstract}
In this paper we will define notion homotopy of morphisms of \ crossed
modules of Lie algebras. Then we construct a groupoid structure of Lie
crossed module morphisms and their homotopies.
\end{abstract}

\maketitle

\section{Introduction}

Crossed modules were firstly introduced by J.H.C Whitehead in
his work on combinatorial homotopy theory \cite{wh}. They have found
important role in many areas of mathematics (including homotopy theory,
homology and cohomology of groups, algebraic K-theory, cyclic homology,
combinatorial group theory and differential geometry). Kassel and Loday \cite%
{KL} introduced crossed modules of Lie algebras as computational algebraic
objects equivalent to simplicial Lie algebras with associated Moore complex
of length 1.

The homotopy relation between crossed module morphisms $\mathcal{P}%
\longrightarrow \mathcal{P}^{\prime }$ can be equivalently addressed either
by considering natural functorial path objects for $\mathcal{P}^{\prime }$
or cylindir objects for $\mathcal{P}$. It yields, given any two crossed
modules $\mathcal{P}$ and $\mathcal{P}^{\prime }$, a groupoid of morphisms $%
\mathcal{P}\longrightarrow \mathcal{P}^{\prime }$ and their homotopies. In
addition , the homotopy relation between crossed module morphisms $\mathcal{P%
}\longrightarrow \mathcal{P}^{\prime }$ is an equivalence relation in the
general case, with no restriction on $\mathcal{P}$ or $\mathcal{P}^{\prime }$%
. This should be compared with what would be guaranteed from the model
category \cite{DS} point of wiev, where we would expect homotopy of maps $%
\mathcal{P}\longrightarrow \mathcal{P}^{\prime }$ to be an equivalence
relation only when $\mathcal{P}$ is a cofibrant (given that any object is
fibrant). In the well-known model category structure in the category of
crossed modules \cite{CG}, obteined by transporting the usal model category
structure of the category of simplicial sets, $\mathcal{P}=(\partial
:M\longrightarrow P)$ is cofibrant if and only if $P$ is a free group (\cite%
{Noohi})

Whitehead in \cite{wh} explored homotopies of morphisms of his "homotopy
systems" and this was put in the general context of crossed complexes of
groupoids by Brown and Higgins in \cite{br-hig}. In this paper we will
define notion of homotopy for morphisms of \ crossed modules of Lie algebras
and we will show that if $\mathcal{P}$ and $\mathcal{P}^{\prime }$ are
crossed modules of Lie algebras, without any restriction on $\mathcal{P}$ or
$\mathcal{P}^{\prime }$, then we have a gropoid of crossed module morphisms $%
\mathcal{P}\longrightarrow \mathcal{P}^{\prime }$ and their homotopies,
similarly to the group case.

\section{Crossed Modules}

J.H.C Whitehead (1949) \cite{wh} described crossed modules in various
contexts especially in his investigations into the algebraic structure of
relative homotopy groups. In this section, we recall the definition of \
crossed modules of \ Lie algebras given by Kassel and Loday \cite{KL}.

Let $M$ and $P$ be two Lie algebras. By an action of $P$ on $M$ we mean a
bilinear map $P\times M\longrightarrow M$, $(p,m)\longmapsto p\cdot m$
satisfying%
\begin{equation*}
\lbrack p,p^{\prime }]\cdot m=p\cdot (p^{\prime }\cdot m)-p^{\prime }\cdot
(p\cdot m)
\end{equation*}%
\begin{equation*}
p\cdot \lbrack m,m^{\prime }]=[p\cdot m,m^{\prime }]+[m,p\cdot m^{\prime }]
\end{equation*}%
for all $m,m^{\prime }\in M$, $p,p^{\prime }\in P.$ For instance, if $P$ is
a subalgebra of some Lie algebra $Q$ (including possibly the case $P=Q$)$,$%
and if $M$ is an ideal in $Q$, then Lie bracket in $Q$ yields an action of $%
P $ on $M.$

A crossed module of Lie algebras is a Lie homomorphism $\partial
:M\longrightarrow P$ together with an action of $P$ on $M$ such that for all
$m,m^{\prime }\in M$, $p\in P$

\begin{equation*}
\text{\textbf{CM1) }}\partial (p\cdot m)=[p,\partial (m)]\text{ }
\end{equation*}%
and%
\begin{equation*}
\text{\textbf{CM2)}}\mathbf{\ \ }\partial m\cdot m^{\prime }=[m,m^{\prime }]%
\text{ }.
\end{equation*}%
The last condition is called the Peiffer identity. We denote such a crossed
module by $\mathcal{P}=(M,P,\partial )$.

A morphism of crossed modules of Lie algebras from $\mathcal{P}=(M,P,\partial )$
to $\mathcal{P}^{\prime }=(M^{\prime },P^{\prime },\partial ^{\prime })$ is
a pair of Lie\textbf{\ }algebra morphisms,%
\begin{equation*}
\theta :M\longrightarrow M^{\prime }\text{, }\psi :P\longrightarrow
P^{\prime }
\end{equation*}%
such that%
\begin{equation*}
\theta (p\cdot m)=\psi (p)\cdot \theta (m)\text{ and }\partial ^{\prime
}\theta (m)=\psi \partial (m).
\end{equation*}%
Therefore we can define the category of Lie crossed modules denoting it as
\textbf{LXmod.}

\subsection{Examples}

\textbf{1.} Let $I$ be any ideal of a Lie algebra $P$. Consider an inclusion
map%
\begin{equation*}
inc:I\longrightarrow P.
\end{equation*}%
Then $(I,P,inc)$ is a crossed module. Conversely given any crossed module $%
\partial :M\longrightarrow P$, one can easily verify that $\partial M=I$ is
an ideal in $P$.

\textbf{2.} Let $M$ be any $R$-bimodule. It can be consider as an $R$%
-algebra with zero multiplication, and then $\mathbf{0}:M\longrightarrow R$
is a crossed $R$-module by $\mathbf{0}(c)\cdot c^{\prime }=0c^{\prime
}=0=cc^{\prime }$ for all $c,c^{\prime }\in M.$

Conversely, given any crossed module $\partial :C\longrightarrow R$, then $%
Ker\partial $ is an $R/\partial C$-module.

\section{Homotopies of Lie crossed module morphisms}

\bigskip Whitehead in \cite{wh} explored homotopies of morphisms of his \
"homotopy systems" and this was put in the general context of crossed
complexes of groupoids by Brown and Higgins in \cite{br-hig}. In this
section we define notion of homotopy for morphisms of \ crossed modules of
Lie algebras.

\

\begin{definition}
Let $\mathcal{P}=(M,P,\partial )$ and $\mathcal{P}^{\prime }=(M^{\prime
},P^{\prime },\partial ^{\prime })$ be crossed modules of Lie algebras, $%
f=(f_{1},f_{0})$ and $g=(g_{1},g_{0})$ be crossed module morphisms $\mathcal{%
P}\longrightarrow $ $\mathcal{P}^{\prime }$. If $\ $there is a linear map $d:P\longrightarrow M^{\prime }$
such that,
\begin{equation*}
\begin{array}{ccc}
g_{0}(p) & = & f_{0}(p)+\partial ^{\prime }d(p) \\
g_{1}(m) & = & f_{1}(m)+d\partial (m)%
\end{array}%
\end{equation*}%
for $m\in M$ and $p\in P,$ then we say that $d$ is a homotopy connecting $f$
to $g$ and we write $d:f\simeq g$ or $f\overset{d}{\longrightarrow }g.$
\end{definition}

\begin{definition}
Let $\mathcal{P}=(M,P,\partial )$ and $\mathcal{P}^{\prime }=(M^{\prime
},P^{\prime },\partial ^{\prime })$ be crossed modules of Lie algebras and $%
f=(f_{1},f_{0})$ be a crossed module morphism $\mathcal{P}=(M,P,\partial
)\longrightarrow $ $\mathcal{P}^{\prime }=(M^{\prime },P^{\prime },\partial
^{\prime }).$ Then a linear map $d:P\longrightarrow M^{\prime }$ satisfying
for all $p,p^{\prime }\in P$
\begin{equation*}
d[p,p^{\prime }]=f_{0}(p)\cdot d(p^{\prime })-f_{0}(p^{\prime })\cdot
d(p)+[d(p),d(p^{\prime })]
\end{equation*}%
is called an $f_{0}$-derivation .
\end{definition}

\begin{proposition}
Let $\mathcal{P}=(M,P,\partial )$ and $\mathcal{P}^{\prime }=(M^{\prime
},P^{\prime },\partial ^{\prime })$ be crossed modules of Lie algebras and $%
f=(f_{1},f_{0})$ be a crossed module morphism $\mathcal{P}\longrightarrow $ $%
\mathcal{P}^{\prime }.$ If $d:P\longrightarrow M^{\prime }$ is an $f_{0}$%
-derivation, then the maps $g_{0}:P\longrightarrow P^{\prime }$ defined by $%
g_{0}(p)=f_{0}(p)+\partial ^{\prime }d(p)$ for all $p\in P$ and $%
g_{1}:M\longrightarrow M^{\prime }$ defined by $g_{1}(m)=f_{1}(m)+d\partial
(m)$ for all $m\in M$ are Lie algebra morphisms.
\end{proposition}

\begin{proof}
For $p,p^{\prime }\in P$,
\begin{equation*}
\begin{array}{lll}
g_{0}(p+p^{\prime }) & = & f_{0}(p+p^{\prime })+\partial ^{\prime
}d(p+p^{\prime }) \\
& = & f_{0}(p)+f_{0}(p^{\prime })+\partial ^{\prime }(d(p)+d(p^{\prime }))
\\
& = & f_{0}(p)+f_{0}(p^{\prime })+\partial ^{\prime }d(p)+\partial ^{\prime
}d(p^{\prime }) \\
& = & g_{0}(p)+g_{0}(p^{\prime }),%
\end{array}%
\end{equation*}%
and%
\begin{equation*}
\begin{array}{lll}
g_{0}[p,p^{\prime }] & = & f_{0}[p,p^{\prime }]+\partial ^{\prime
}d[p,p^{\prime }] \\
& = & [f_{0}(p),f_{0}(p^{\prime })]+\partial ^{\prime }(f_{0}(p)\cdot
d(p^{\prime })-f_{0}(p^{\prime })\cdot d(p) \\
&  & +[d(p),d(p^{\prime })] \\
& = & [f_{0}(p),f_{0}(p^{\prime })]+\partial ^{\prime }(f_{0}(p)\cdot
d(p^{\prime }))-\partial ^{\prime }(f_{0}(p^{\prime })\cdot d(p)) \\
&  & +\partial ^{\prime }[d(p),d(p^{\prime })] \\
& = & [f_{0}(p),f_{0}(p^{\prime })]+[f_{0}(p),\partial ^{\prime }d(p^{\prime
})]-[f_{0}(p^{\prime }),\partial ^{\prime }d(p)] \\
&  & +[\partial ^{\prime }d(p),\partial ^{\prime }d(p^{\prime })] \\
& = & [f_{0}(p),f_{0}(p^{\prime })]+[f_{0}(p),\partial ^{\prime }d(p^{\prime
})]+[\partial ^{\prime }d(p),f_{0}(p^{\prime }),] \\
&  & +[\partial ^{\prime }d(p),\partial ^{\prime }d(p^{\prime })] \\
& = & [f_{0}(p)+\partial ^{\prime }d(p),f_{0}(p^{\prime })+\partial ^{\prime
}d(p^{\prime })], \\
& = & [g_{0}(p),g_{0}(p^{\prime })].%
\end{array}%
\end{equation*}%
Thus $g_{0}$ is a Lie algebra morphism.

For $m,m^{\prime }\in M$ $,$%
\begin{equation*}
\begin{array}{lll}
g_{1}(m+m^{\prime }) & = & f_{1}(m+m^{\prime })+d\partial (m+m^{\prime }) \\
& = & f_{1}(m)+f_{1}(m^{\prime })+d(\partial (m)+\partial (m^{\prime })) \\
& = & f_{1}(m)+f_{1}(m^{\prime })+d\partial (m)+d\partial (m^{\prime }) \\
& = & g_{1}(m)+g_{1}(m^{\prime }),%
\end{array}%
\end{equation*}%
and%
\begin{equation*}
\begin{array}{lll}
g_{1}[m,m^{\prime }] & = & f_{1}[m,m^{\prime }]+d\partial \lbrack
m,m^{\prime }] \\
& = & [f_{1}(m),f_{1}(m^{\prime })]+d[\partial (m),\partial (m^{\prime })]
\\
& = & [f_{1}(m),f_{1}(m^{\prime })]+f_{0}(\partial (m))\cdot d(\partial
(m^{\prime })) \\
&  & -f_{0}(\partial (m^{\prime }))\cdot d(\partial (m))+[d(\partial
(m)),d(\partial (m^{\prime }))] \\
& = & [f_{1}(m),f_{1}(m^{\prime })]+\partial ^{\prime }f_{1}(m)\cdot
d(\partial (m^{\prime })) \\
&  & -\partial ^{\prime }f_{1}(m^{\prime })\cdot d(\partial (m))+[d(\partial
(m)),d(\partial (m^{\prime }))] \\
& = & [f_{1}(m),f_{1}(m^{\prime })]+[f_{1}(m),d(\partial (m^{\prime
}))]-[f_{1}(m^{\prime }),d(\partial (m))] \\
&  & +[d(\partial (m)),d(\partial (m^{\prime }))] \\
& = & [f_{1}(m),f_{1}(m^{\prime })]+[f_{1}(m),d(\partial (m^{\prime
}))]+[d(\partial (m)),f_{1}(m^{\prime })] \\
&  & +[d(\partial (m)),d(\partial (m^{\prime }))]+[d(\partial
(m)),d(\partial (m^{\prime }))] \\
& = & [f_{1}(m)+d(\partial (m)),f_{1}(m^{\prime })+d(\partial (m^{\prime }))]
\\
& = & [g_{1}(m),g_{1}(m^{\prime })]%
\end{array}%
\end{equation*}%
thus $g_{1}$ is a Lie algebra morphism.
\end{proof}

\begin{proposition}
\bigskip
\begin{equation*}
g_{0}\partial =\partial ^{\prime }g_{1}.
\end{equation*}
\end{proposition}

\begin{proof}
For $m\in M$,%
\begin{equation*}
\begin{array}{lll}
(g_{0}\partial )(m) & = & g_{0}(\partial (m)) \\
& = & f_{0}(\partial (m))+\partial ^{\prime }d(\partial (m)) \\
& = & \partial ^{\prime }(f_{1}(m))+\partial ^{\prime }(d\partial (m)) \\
& = & \partial ^{\prime }(f_{1}(m)+d\partial (m)) \\
& = & \partial ^{\prime }(g_{1}(m)) \\
& = & (\partial ^{\prime }g_{1})(m),%
\end{array}%
\end{equation*}%
so%
\begin{equation*}
g_{0}\partial =\partial ^{\prime }g_{1}.
\end{equation*}
\end{proof}

\begin{proposition}
For $p\in P$ and $m\in M,$%
\begin{equation*}
g_{1}(p\cdot m)=g_{0}(p)\cdot g_{1}(m)\text{.}
\end{equation*}
\end{proposition}

\begin{proof}
\begin{equation*}
\begin{array}{lll}
g_{1}(p\cdot m) & = & f_{1}(p\cdot m)+d\partial (p\cdot m) \\
& = & f_{0}(p)\cdot f_{1}(m)+d[p,\partial (m)] \\
& = & f_{0}(p)\cdot f_{1}(m)+f_{0}(p)\cdot d\partial (m)-f_{0}\partial
(m)\cdot d(p)+[d(p),d\partial (m)] \\
& = & f_{0}(p)\cdot f_{1}(m)+f_{0}(p)\cdot d\partial (m)-\partial ^{\prime
}f_{1}(m)\cdot d(p)+[d(p),d\partial (m)] \\
& = & f_{0}(p)\cdot f_{1}(m)+f_{0}(p)\cdot d\partial
(m)-[f_{1}(m),d(p)]+[d(p),d\partial (m)] \\
& = & f_{0}(p)\cdot f_{1}(m)+f_{0}(p)\cdot d\partial
(m)+[d(p),f_{1}(m)]+[d(p),d\partial (m)] \\
& = & f_{0}(p)\cdot f_{1}(m)+f_{0}(p)\cdot d\partial
(m)+[d(p),f_{1}(m)+d\partial (m)] \\
& = & f_{0}(p)\cdot f_{1}(m)+f_{0}(p)\cdot d\partial (m)+\partial ^{\prime
}d(p)\cdot (f_{1}(m)+d\partial (m)) \\
& = & g_{0}(p)\cdot g_{1}(m).%
\end{array}%
\end{equation*}
\end{proof}

Thus by the above propositions, we can give the following theorem.

\begin{theorem}
Let $\mathcal{P}=(M,P,\partial )$ and $\mathcal{P}^{\prime }=(M^{\prime
},P^{\prime },\partial ^{\prime })$ be crossed modules of Lie algebras and $%
f=(f_{1},f_{0})$ be a crossed module morphism $\mathcal{P}\longrightarrow $ $%
\mathcal{P}^{\prime }.$ If $d:P\longrightarrow M^{\prime }$ is an $f_{0}$%
-derivation, then the map $g=(g_{1},g_{0}):\mathcal{P}\longrightarrow
\mathcal{P}^{\prime }$ is a Lie crossed module morphism.
\end{theorem}

\begin{corollary}
Let $\mathcal{P}=(M,P,\partial )$ and $\mathcal{P}^{\prime }=(M^{\prime
},P^{\prime },\partial ^{\prime })$ be crossed modules of Lie algebras and $%
f=(f_{1},f_{0}),g=(g_{1},g_{0})$ be crossed module morphisms $\mathcal{P}%
\longrightarrow $ $\mathcal{P}^{\prime }.$ Then the $f_{0}$-derivation $%
d:P\longrightarrow M^{\prime }$ satisfying for all $p,p^{\prime }\in P$
\begin{equation*}
d[p,p^{\prime }]=f_{0}(p)\cdot d(p^{\prime })-f_{0}(p^{\prime })\cdot
d(p)+[d(p),d(p^{\prime })]
\end{equation*}%
is a homotopy connecting $f$ to $g.$
\end{corollary}

\section{Groupoid Structure for Crossed module morphisms
and their homotopies}

In this section we construct a groupoid structure whose objects are the
crossed module morphisms $\mathcal{P}\longrightarrow \mathcal{P}^{\prime }$,
with morphisms being the homotopies between them.

\begin{lemma}
Let $\mathcal{P}=(M,P,\partial )$ and $\mathcal{P}^{\prime }=(M^{\prime
},P^{\prime },\partial ^{\prime })$ be crossed modules of Lie algebras and $%
f=(f_{1},f_{0})$ be a crossed module morphism $\mathcal{P}\longrightarrow $ $%
\mathcal{P}^{\prime }.$ Then the null function $0:P\longrightarrow M^{\prime
}$ , $0(p)=0_{M^{\prime }}$ defines an $f_{0}$ derivation connecting $f$ to $%
f.$
\end{lemma}

\begin{lemma}
Let $f=(f_{1},f_{0})$ and $g=(g_{1},g_{0})$ be crossed module morphisms $%
\mathcal{P}\longrightarrow $ $\mathcal{P}^{\prime }$ and $d$ be an $f_{0}$%
-derivation connecting $f$ to $g$. Then the linear map $\overline{d}%
=-d:P\longrightarrow M^{\prime }$ with $\overline{d}(p)=-d(p)$ is a $g_{0}$
derivation connecting $g$ to $f.$
\end{lemma}

\begin{proof}
Since $d$ is an $f_{0}$ derivation connecting $f$ to~$g$, we have
\begin{equation}
f_{0}(p)=g_{0}(p)+\partial ^{\prime }\overline{d}(p)\text{ and }%
f_{1}(m)=g_{1}(m)+\overline{d}\partial (m).
\end{equation}%
Moreover $\overline{d}$ is a $g_{0}$ derivation, since:%
\begin{equation*}
\begin{array}{lll}
\overline{d}[p,p^{\prime }] & = & -d[p,p^{\prime }] \\
& = & -(f_{0}(p)\cdot d(p^{\prime })-f_{0}(p^{\prime })\cdot
d(p)+[d(p),d(p^{\prime })]) \\
& = & -f_{0}(p)\cdot d(p^{\prime })+f_{0}(p^{\prime })\cdot
d(p)-[d(p),d(p^{\prime })] \\
&  & -[d(p),d(p^{\prime })]+[d(p),d(p^{\prime })] \\
& = & -f_{0}(p)\cdot d(p^{\prime })-\partial ^{\prime }d(p)\cdot d(p^{\prime
})+f_{0}(p^{\prime })\cdot d(p) \\
&  & +\partial ^{\prime }d(p^{\prime })\cdot d(p)+[d(p),d(p^{\prime })] \\
& = & f_{0}(p)\cdot (-d(p^{\prime }))+\partial ^{\prime }d(p)\cdot
(-d(p^{\prime }))-f_{0}(p^{\prime })\cdot (-d(p)) \\
&  & -\partial ^{\prime }d(p^{\prime })\cdot (-d(p))+[-d(p),-d(p^{\prime })]
\\
& = & (f_{0}(p)+\partial ^{\prime }d(p))\cdot (-d(p^{\prime
}))-(f_{0}(p^{\prime })+\partial ^{\prime }d(p^{\prime }))\cdot (-d(p)) \\
&  & +[-d(p),-d(p^{\prime })] \\
& = & g_{0}(p)\cdot (-d(p^{\prime }))-g_{0}(p^{\prime })\cdot
(-d(p))+[-d(p),-d(p^{\prime })].%
\end{array}%
\end{equation*}
\end{proof}

\begin{lemma}
(Concatenation of derivations) Let $f,g$ and $h$ be Lie crossed module
morphisms $\mathcal{P}\longrightarrow $ $\mathcal{P}^{\prime }$, $d$ be an $%
f_{0}$ derivation connecting $f$ to $g$, and $d^{\prime }$ be a $g_{0}$
derivation connecting $g$ to $h$. Then the linear map $(d+d^{\prime
}):P\longrightarrow M^{\prime }$ such that $(d+d^{\prime
})(p)=d(p)+d^{\prime }(p),$ defines an $f_{0}$ derivation (therefore a
homotopy) connecting $f$ to $h.$
\end{lemma}

\begin{proof}
We know that $f\overset{(f_{0},d)}{\longrightarrow }g$ and $g\overset{%
(g_{0},d^{\prime })}{\longrightarrow }h$. Therefore by definition%
\begin{equation*}
\begin{array}{ccc}
h_{0}(p) & = & f_{0}(p)+\partial ^{\prime }(d+d^{\prime })(p) \\
h_{1}(m) & = & f_{1}(m)+(d+d^{\prime })\partial (m).%
\end{array}%
\end{equation*}%
Let us see that $d+d^{\prime }$ satisfies the condition for it to be an $%
f_{0}$ derivation:%
\begin{equation*}
\begin{array}{lll}
(d+d^{\prime })[p,p^{\prime }] & = & d[p,p^{\prime }]+d^{\prime
}[p,p^{\prime }] \\
& = & f_{0}(p)\cdot d(p^{\prime })-f_{0}(p^{\prime })\cdot
d(p)+[d(p),d(p^{\prime })] \\
&  & +g_{0}(p)\cdot d^{\prime }(p^{\prime })-g_{0}(p^{\prime })\cdot
d^{\prime }(p)+[d^{\prime }(p),d^{\prime }(p^{\prime })] \\
& = & f_{0}(p)\cdot d(p^{\prime })-f_{0}(p^{\prime })\cdot
d(p)+[d(p),d(p^{\prime })] \\
&  & +(f_{0}(p)+\partial ^{\prime }d(p))\cdot d^{\prime }(p^{\prime
})-(f_{0}(p^{\prime })+\partial ^{\prime }d(p^{\prime }))\cdot d^{\prime }(p)
\\
&  & +[d^{\prime }(p),d^{\prime }(p^{\prime })] \\
& = & f_{0}(p)\cdot d(p^{\prime })-f_{0}(p^{\prime })\cdot
d(p)+[d(p),d(p^{\prime })]+f_{0}(p)\cdot d^{\prime }(p^{\prime }) \\
&  & +\partial ^{\prime }d(p)\cdot d^{\prime }(p^{\prime })-f_{0}(p^{\prime
})\cdot d^{\prime }(p)-\partial ^{\prime }d(p^{\prime })\cdot d^{\prime }(p)
\\
&  & +[d^{\prime }(p),d^{\prime }(p^{\prime })] \\
& = & f_{0}(p)\cdot d(p^{\prime })-f_{0}(p^{\prime })\cdot
d(p)+[d(p),d(p^{\prime })]+f_{0}(p)\cdot d^{\prime }(p^{\prime }) \\
&  & +[d(p),d^{\prime }(p^{\prime })]-f_{0}(p^{\prime })\cdot d^{\prime
}(p)-[d(p^{\prime }),d^{\prime }(p)]+[d^{\prime }(p),d^{\prime }(p^{\prime
})] \\
& = & f_{0}(p)\cdot (d(p^{\prime })+d^{\prime }(p^{\prime
}))-f_{0}(p^{\prime })\cdot (d(p)+d^{\prime }(p)) \\
&  & +[d(p)+d^{\prime }(p),d(p^{\prime })+d^{\prime }(p^{\prime })]%
\end{array}%
\end{equation*}%
for all $p,p^{\prime }\in P.$ Therefore $(d+d^{\prime })$ is an $f_{0}$
derivation connecting $f$ to~$h.$
\end{proof}

\begin{theorem}
Let $\mathcal{P}$ and $\mathcal{P}^{\prime }$ be two arbitrary crossed
modules of Lie algebras. We have a groupoid $HOM(\mathcal{P},\mathcal{P}%
^{\prime })$, whose objects are the crossed module morphisms $\mathcal{P}%
\longrightarrow $ $\mathcal{P}^{\prime },$ the morphisms being their
homotopies. In particular the relation below, for crossed module morphisms $%
\mathcal{P}\longrightarrow $ $\mathcal{P}^{\prime },$ is an equivalence
relation:%
\begin{equation*}
\text{\textquotedblleft }f\simeq g\Leftrightarrow \text{there exists an }%
f_{0}\text{ derivation }d\text{ connecting }f\text{ to }g\text{%
\textquotedblright .}
\end{equation*}
\end{theorem}

\section{Acknowledgements}

This research was supported by Eski\c{s}ehir Osmangazi University Scientific Research Center (BAP) under Grant No:2016-1129.


\begin{thebibliography}{9}
\bibitem{br-hig} \textsc{Brown, R. and Higgins P. J., } Tensor Products and
Homotopies for $\omega -$groupoids and crossed complexes, \emph{Journal of
Pure and Applied Algebra } \textbf{47}, (1987), 1-33.

\bibitem{CG} \textsc{Cabello, J.G. and Garzon A.R. } Closed model structures
for algebraic models of n-types, \emph{Journal of Pure and Applied Algebra }
103\textbf{\ (3)}, (1995), 287--302.

\bibitem{DS} \textsc{Dwyer, W.G. - Spalinski, J. }Homotopy theories and
model categories, \ In \emph{Handbook of algebraic topology,} \ pages
73-126. Amsterdam: Nort Holland,\ (1995).

\bibitem{KL} \textsc{Kasel, C.} and \textsc{Loday, J.L. }Extensions
centrales d'algebres de Lie. \ \emph{Ann. Inst. Fourier (Grenoble),} \
\textbf{33,}\ (1982) 119-142.

\bibitem{Noohi} \textsc{Noohi, B. } Notes on 2-groupoids, 2-groups and
crossed modules, \emph{Homology Homotopy Appl. } 9 (1), (2007), 75-106.

\bibitem{wh} \textsc{Whitehead, J.H.C.} Combinatorial Homotopy I and II,
\emph{Bull. Amer. Math. Soc., }\ \textbf{55}, 231-245 and 453-456 (1949).
\end{thebibliography}
\end{document}